\documentclass[a4paper,11pt]{article}
\usepackage[T1]{fontenc}
\usepackage{amsfonts}
\usepackage{verbatim}
\usepackage{amsmath, amsthm, amssymb}

\title{A Note on the Multivariate CLT and Convergence of L\'evy Processes at Long and Short Times}
\author{Michael Grabchak\footnote{Email address: mgrabcha@uncc.edu}, {\it University of North Carolina Charlotte}}

\begin{document}
\newtheorem{prop}{Proposition}
\newtheorem{thrm}{Theorem}
\newtheorem{defn}{Definition}
\newtheorem{cor}{Corollary}
\newtheorem{lemma}{Lemma}

\newcommand{\rd}{\mathrm d}
\newcommand{\rE}{\mathrm E}
\newcommand{\ts}{TS^p_\alpha}
\newcommand{\tr}{\mathrm{tr}}
\newcommand{\iid}{\stackrel{\mathrm{iid}}{\sim}}
\newcommand{\eqd}{\stackrel{d}{=}}
\newcommand{\cond}{\stackrel{d}{\rightarrow}}
\newcommand{\conv}{\stackrel{v}{\rightarrow}}
\newcommand{\conw}{\stackrel{w}{\rightarrow}}
\newcommand{\conp}{\stackrel{p}{\rightarrow}}
\newcommand{\confdd}{\stackrel{fdd}{\rightarrow}}
\newcommand{\plim}{\mathop{\mathrm{p\mbox{-}lim}}}
\newcommand{\dlim}{\operatorname*{d-lim}}

\maketitle

\begin{abstract}

We show that a necessary and sufficient condition for the sum of iid random vectors to converge (under appropriate shifting and scaling) to a multivariate Gaussian distribution is that the truncated second moment matrix is slowly varying at infinity.  This is more natural than the standard conditions, and allows for the possibility that the limiting Gaussian distribution is degenerate (so long as it is not concentrated at a point). We also give necessary and sufficient conditions for a $d$-dimensional L\'evy process to converge (under appropriate shifting and scaling) to a multivariate Gaussian distribution as time approaches zero or infinity.\\


\end{abstract}

\section{Introduction}

Let $X$ be a random variable with distribution $\mu$ on $\mathbb R^d$. If $d=1$, a necessary and sufficient condition for $X$ to belong to the domain of attraction of the Gaussian is that the truncated second moment function
$$
f(t) = \int_{|x|\le t} x^2 \mu(\rd x)
$$
is slowly varying (see e.g.\ \cite{Feller:1971}).  However, we have found no similar condition in the literature for the case $d\ge2$.  Instead, conditions are given in terms of regular variation of certain quadratic forms, see \cite{Rvaceva:1962}. For this reason, conditions are only given for convergence to nondegenerate Gaussian distributions.  We give necessary and sufficient conditions for $X$ to belong to the domain of attraction of any (possibly degenerate) Gaussian distribution in terms of the slow variation of its truncated second moment matrix. We also give the corresponding result for the long and short time behavior of L\'evy processes. 

A L\'evy process is a continuous time process that generalizes summation of iid random variables, see \cite{Sato:1999}. Assume that $\mu$ is an infinitely divisible distribution and let $\{X_t:t\ge0\}$ be a L\'evy processes with $X_1\sim\mu$. The long (short) time behavior of this process is the weak limit of $X_t$, under appropriate shifting and scaling, as $t$ approaches infinity (zero).  An alternate, but equivalent, definition, in terms of weak convergence of certain time rescaled processes, is also sometimes used (see e.g.\ \cite{Rosinski:2007}).  Necessary and sufficient conditions for the long and short time behavior to be an infinite variance stable distribution are given in \cite{Grabchak:2013}.  In this paper, we give necessary and sufficient conditions for the long and short time behavior of the process to be Gaussian.

In the next section we introduce our notation and give some background.  In particular we define regular variation for matrix-valued functions.  While regular variation of invertible matrix-valued functions has been studied (see e.g. \cite{Balkema:1973} or \cite{Meerschaert:Scheffler:2001}), our definition, which is valid even for certain non-invertible matrix-valued functions, appears to be new. In Section 3 we state our main results and in Section \ref{sec: Proofs} we give the proofs.

\section{Preliminaries}

Let $\mathbb R^d$ be the space of $d$-dimensional column vectors of real numbers, let $|\cdot|$ be the usual norm on $\mathbb R^d$, and let $\mathfrak B(\mathbb R^d)$ denote the Borel sets on $\mathbb R^d$.  For $x\in\mathbb R^d$ we write $x=(x_1,x_2,\dots,x_d)$ and let $x^T$ be the transpose of $x$.  We write $X\sim\mu$ to denote that $X$ is a random variable on $\mathbb R^d$ with distribution $\mu$. For a sequence $X_1,X_2,\dots$ of random variables we write $\dlim X_n$ to denote the limit in distribution. Let $\mathbb R^{d\times d}$ be the collection of all $d\times d$ dimensional matrices with real entries. If $A\in\mathbb R^{d\times d}$ let $\tr A$ be the trace of the matrix $A$. If $f$ and $g$ are real-valued functions and $c\in\{0,\infty\}$, we write $f(t)\sim g(t)$ as $t\to c$ to denote $f(t)/g(t)\to 1$ as $t\to c$.  When dealing with infinity we adopt the following conventions: $1/\infty=0$ and $1/0=\infty$.

For $c\in\{0,\infty\}$ and $\rho\in\mathbb R$, a Borel function $f:(0,\infty)\mapsto(0,\infty)$ is called regularly varying at $c$ with index $\rho$ if
$$
\lim_{x\rightarrow c} \frac{f(tx)}{f(x)}=t^{\rho}.
$$
In this case we write $f\in RV^c_\rho$. If $h(x) = 1/f(1/x)$ then 
\begin{eqnarray}\label{eq: zero equiv inf}
f\in RV^c_\rho \mbox{ if and only if } h\in RV^{1/c}_\rho.
\end{eqnarray}
If $f\in RV^c_\rho$ then $f(x+t)\sim f(t)$ as $t\to c$ (this follows from the Potter Bounds, see e.g.\ Theorem 1.5.6 in \cite{Bingham:Goldie:Teugels:1987}).
If $f\in RV^c_\rho$ with $\rho>0$ and $f^{\leftarrow}(x) = \inf\left\{ y>0: f(y)>x \right\}$ then 
\begin{eqnarray*}
f^{\leftarrow}\in RV_{1/\rho}^c
\end{eqnarray*}
and $f^\leftarrow$ is an asymptotic inverse of $f$ in the sense that
\begin{eqnarray*}
f(f^{\leftarrow}(x)) \sim f^{\leftarrow}(f(x)) \sim x \ \mbox{as} \ x\rightarrow c.
\end{eqnarray*}
When $c=\infty$ this result is given on page 28 of \cite{Bingham:Goldie:Teugels:1987}. The case when $c=0$ can be shown using an extension of those results and \eqref{eq: zero equiv inf}.

We  now introduce a definition of regular variation for matrix valued functions.  While, regular variation of invertible matrix-valued functions is defined in \cite{Balkema:1973} and \cite{Meerschaert:Scheffler:2001}, we need a different definition to allow for regular variation of certain non-invertible matrix-valued functions. 

\begin{defn}
Fix $c\in\{0,\infty\}$ and $\rho\in\mathbb R$. Let $A_\bullet:\mathbb (0,\infty)\mapsto \mathbb R^{d\times d}$ and $B\in \mathbb R^{d\times d}$.
If $\tr A_\bullet \in RV^c_\rho$ and
\begin{eqnarray*}
\lim_{t\rightarrow c}\frac{A_t}{\tr A_t} = B
\end{eqnarray*}
we say that $A_\bullet$ is matrix regularly varying at $c$ with index $\rho$ and limiting matrix $B$. In this case we write $A_\bullet\in MRV^c_\rho(B)$.
\end{defn}

Before proceeding, we review some basic properties of infinitely divisible distributions. Recall that the characteristic function of an infinitely divisible distribution $\mu$ can be written as $\hat\mu(z) = \exp\{C_{\mu}(z)\}$ where
\begin{eqnarray*}
C_{\mu}(z) = -\frac{1}{2}\langle z,Az\rangle + i\langle b,z\rangle + \int_{\mathbb R^d}\left(e^{i\langle z,x\rangle}-1-i\frac{\langle z,x\rangle}{1+|x|^2}\right)M(\rd x)
\end{eqnarray*}
is the cumulant generating function of $\mu$. Here $A$ is a symmetric nonnegative-definite $d\times d$ matrix, $b\in\mathbb R^d$, and $M$ is a L\'evy measure, i.e.\ $M(\{0\})=0$ and 
\begin{eqnarray}\label{eq: integ cond}
\int_{\mathbb R^d}\left(|x|^2\wedge1\right)M(\rd x)<\infty. 
\end{eqnarray}
The measure $\mu$ is uniquely identified by the so called L\'evy triplet $(A,M,b)$ and we write $\mu=ID(A,M,b)$. For more information about infinitely divisible distributions and their associated L\'evy processes see \cite{Sato:1999}.

\section{Main Results}

First, we give necessary and sufficient conditions for the multivariate central limit theorem.

\begin{thrm}\label{thrm: doa normal for CLT}
Let $B\ne0$ be a symmetric nonnegative-definite $d\times d$ matrix, let $\mu$ be a probability measure on $\mathbb R^d$, let $X_1,X_2,\dots\iid \mu$, and let
\begin{eqnarray}\label{eq: trunc second moment func CLT}
A_t = \int_{|x|\le t} xx^T \mu(\rd x)- \int_{|x|\le t} x \mu(\rd x)\int_{|x|\le t} x^T \mu(\rd x).
\end{eqnarray}
There exist non-stochastic $a_n\in(0,\infty)$ and $\xi_n\in\mathbb R^d$ such that
\begin{eqnarray}\label{eq: long time for gaus CLT}
a_n \sum_{i=1}^n X_i-\xi_n \cond N(0,B) \ \mbox{as}\ n\rightarrow \infty
\end{eqnarray}
if and only if $A_\bullet\in MRV^\infty_0(kB)$, where $k=1/\tr B$. Moreover, when this holds $a_\bullet\in RV^c_{-1/2}$ and if $\int_{\mathbb R^d}|x|^2\mu(\rd x)<\infty$ then
\begin{eqnarray}\label{eq: a CLT finite var}
a_n \sim k^{-1/2}\left[\int_{\mathbb R^d}|x|^2\mu(\rd x)-\int_{\mathbb R^d}x^T\mu(\rd x)\int_{\mathbb R^d}x\mu(\rd x)\right]^{-1/2}n^{-1/2}
\end{eqnarray}
and if $\int_{\mathbb R^d}|x|^2\mu(\rd x)=\infty$ then
\begin{eqnarray}\label{eq: a CLT}
a_n\sim k^{-1/2}/h^{\leftarrow}(n) \mbox{ as } n\rightarrow \infty,
\end{eqnarray}
where
\begin{eqnarray}\label{eq: g for CLT}
h(t)=\frac{t^2}{\int_{|x|\le t} |x|^2 \mu(\rd x)}.
\end{eqnarray}
\end{thrm}

\begin{cor}\label{cor: nec and suf CLT}
Let $\mu$ be a probability measure on $\mathbb R^d$ and let
$$
A_t'=\int_{|x|\le t} xx^T \mu(\rd x).
$$
Then $\mu$ belongs to the domain of attraction of some multivariate normal distribution if and only if there exists a nonnegative definite matrix $B'\ne0$ such that $A'_\bullet\in MRV_0^\infty(B')$.
\end{cor}

We now give necessary and sufficient conditions for the long and short time behavior of a L\'evy process to be Gaussian. In one dimension this was characterized in \cite{Doney:Maller:2002}. Since long and short time behavior of Brownian motion is straight-foreward, without loss of generality, we focus on the case where the Gaussian part is zero. 

\begin{thrm}\label{thrm: doa of normal for Levy}
Fix $c\in\{0,\infty\}$. Let $B\ne0$ be a symmetric nonnegative-definite matrix, let $\{X_t:t\ge0\}$ be a L\'evy process with $X_1\sim ID(0,M,b)$, and let
\begin{eqnarray}\label{eq: trunc second moment func Levy}
A_t = \int_{|x|\le t} xx^T M(\rd x).
\end{eqnarray}
There exist non-stochastic functions $a_t$ and $\xi_t$ such that
\begin{eqnarray}\label{eq: long time for gaus Levy}
a_tX_t-\xi_t\cond N(0,B) \ \mbox{as}\ t\rightarrow c
\end{eqnarray}
if and only if $A_\bullet\in MRV^c_0(kB)$ where $k=1/\tr B$. Moreover, when this holds, $a_\bullet\in RV^c_{-1/2}$ and
\begin{eqnarray}\label{eq: a for Levy}
a_t\sim  k^{-1/2}/h^{\leftarrow}(t) \ \mbox{as}\ t\rightarrow c,
\end{eqnarray}
where 
\begin{eqnarray}\label{eq: g for Levy}
h(t)=\frac{t^2}{\int_{|x|\le t} |x|^2 M(\rd x)}.
\end{eqnarray}
\end{thrm}

Combining Corollary \ref{cor: nec and suf CLT} with Theorem 2 gives the following.

\begin{cor}
Let
$\mu=ID(0,M,b)$. There is a nonnegative definite matrix $B\ne0$ with
$$
\int_{|x|\le \bullet} xx^T \mu(\rd x) \in MRV^\infty_{0}(B)
$$
if and only if there is a nonnegative definite matrix $B'\ne0$ with
$$
\int_{|x|\le \bullet} xx^T M(\rd x) \in MRV^\infty_{0}(B').
$$
\end{cor}

\section{Proofs}\label{sec: Proofs}

\begin{lemma}\label{lemma: eqiv between deltas}
Fix $a,b\in(0,\infty)$ and let $\{M_n\}$ be a sequence of measures on $\mathbb R^d$ satisfying \eqref{eq: integ cond}.  If, for any $s>0$,
$
\lim_{n\to\infty}M_n(x:|x|>s)=0 
$  
then
\begin{eqnarray}\label{eq: eqiv Mn}
\lim_{n\rightarrow\infty}\left( \int_{|x|\le a}xx^T M_n(\rd x) - \int_{|x|\le b}xx^T M_n(\rd x) \right)=0.
\end{eqnarray}
\end{lemma}

\begin{proof}
Without loss of generality assume that $a<b$. For all $1\le i,j\le d$ 
\begin{eqnarray*}
&&\left|\int_{|x|\le b} x_ix_j M_n(\rd x)-\int_{|x|\le a} x_ix_j M_n(\rd x)\right| = \left|\int_{a<|x|\le b} x_ix_j M_n(\rd x)\right|\\
&& \le \int_{a<|x|\le b} |x|^2 M_n(\rd x) \le b^2 M_n\left(x:a<|x|\le b\right) \rightarrow 0,
\end{eqnarray*}
as required.
\end{proof}

The following is a specialization of Theorem 3.1.14 in \cite{Meerschaert:Scheffler:2001} (or Theorem 8.7 in \cite{Sato:1999}). To put the result in this form we use Lemma \ref{lemma: eqiv between deltas}.

\begin{lemma}\label{lemma: conv ID}
If $\mu_n=ID(0,M_n,b_n)$ then $\mu_n\conw N(b,A)$ if and only if $b_n\to b$,
\begin{eqnarray}\label{eq: Mn to 0}
M_n(x:|x|>s)\to 0 \mbox{ for every } s>0,
\end{eqnarray}
and
\begin{eqnarray}\label{eq: gaus comp inf div}
\lim_{n\rightarrow\infty} \int_{|x|\le 1} xx^T M_n(\rd x) = A.
\end{eqnarray}
\end{lemma}

\begin{lemma}\label{lemma: at conv to 0 or infty}
Let $\{X_t: t\ge0\}$ be a L\'evy process, let $Y\sim N(0,A)$ with $A\ne0$, and let $a_t$ be a positive function. Assume that \eqref{eq: long time for gaus Levy} holds with some function $\xi_t$.\\
1. If $c=0$ then $\lim_{t\downarrow0}a_t=\infty$ and $a_{1/t}\sim a_{1/(t+1)}$ as $t\rightarrow\infty$.\\
2. If $c=\infty$ then $\lim_{t\rightarrow\infty}a_t=0$ and $a_t\sim a_{t+1}$ as $t\rightarrow\infty$.
\end{lemma}

\begin{proof}
First assume $c=0$. Let $\ell:=\liminf_{t\downarrow0}a_t$ and assume for the sake of contradiction that $\ell<\infty$. This means that there is a sequence of positive real numbers $\{t_n\}$ converging to $0$ such that $\lim_{n\rightarrow\infty} a_{t_n} = \ell$. Consider a further subsequence $\{t_{n_i}\}$ such that $\lim_{t\to\infty} \xi_{t_{n_i}}$ exists (although we allow it to be infinite). Stochastic continuity of L\'evy processes implies that $X_t\conp 0$ as $t\downarrow0$, thus Slutzky's Theorem implies that
$$
Y=\dlim_{i\rightarrow\infty} (a_{t_{n_i}}X_{t_{n_i}}-\xi_{t_{n_i}}) \eqd \ell 0-\lim_{i\to\infty}\xi_{t_{n_i}},
$$
which contradicts the assumption that $Y\sim N(0,A)$. Thus $\lim_{t\downarrow0}a_t=\infty$.

Let $C_{X_1}(z)$, $z\in\mathbb R^d$, be the cumulant generating function of $X_1$. The characteristic function of $a_{1/t}X_{1/t}-\xi_{1/t}$ is $\exp\left(\frac{1}{t}C_{X_1}(a_{1/t}z)-i\langle z,\xi_{1/t}\rangle\right)$. If $\hat\mu_Y(z)$ is the characteristic function of $Y$ then
\begin{eqnarray*}
\hat\mu_Y(z) &=& \lim_{t\rightarrow\infty} \exp\left(\frac{1}{t}C_{X_1}(a_{1/t}z)-i\langle z,\xi_{1/t}\rangle\right)\\
&=& \lim_{t\rightarrow\infty} \exp\left(\frac{1}{t+1}C_{X_1}(a_{1/t}z)-i\langle z, \frac{t}{t+1}\xi_{1/t}\rangle\right).
\end{eqnarray*}
This implies that
\begin{eqnarray*}
Y &\eqd& \dlim_{t\to\infty} \left(a_{1/t}X_{1/t}-\xi_{1/t}\right) \\
&\eqd& \dlim_{t\to\infty} \left(a_{1/t}X_{1/(t+1)}-\frac{t}{t+1}\xi_{1/t}\right)\\
&\eqd& \dlim_{t\to\infty} \left(\frac{a_{1/t}}{a_{1/(t+1)}}\left(a_{1/(t+1)}X_{1/(t+1)}-\xi_{1/(t+1)}\right)\right.\\
&&\qquad\qquad \left.+\frac{a_{1/t}}{a_{1/(t+1)}}\xi_{1/(t+1)}-\frac{t}{t+1}\xi_{1/t}\right).
\end{eqnarray*}
Since $\left(a_{1/(t+1)}X_{1/(t+1)}-\xi_{1/(t+1)}\right)\cond Y$ as $t\to\infty$, the result follows.

Now assume $c=\infty$. The L\'evy measure of $a_tX_t-\xi_t$ is $M_t(\cdot) = tM(\cdot/a_t)$. By Lemma \ref{lemma: conv ID}  for any $s>0$
\begin{eqnarray*}
\lim_{t\rightarrow\infty}tM(|x|>s/a_t) = \lim_{t\rightarrow\infty} M_t(|x|>s) =  0,
\end{eqnarray*}
which implies that $\lim_{t\to\infty}a_t=0$. Now let $X'\eqd X_1$ be independent of $\{X_t:t\ge0\}$ and note that $a_tX'\conp0$ as $t\to\infty$. The facts that
\begin{eqnarray*}
Y &\eqd& \dlim_{t\rightarrow\infty} \left(a_{t+1}X_{t+1}-\xi_{t+1}\right)\\
&\eqd& \dlim_{t\rightarrow\infty} \left(a_{t+1}X_t + a_{t+1}X'-\xi_{t+1}\right)\\
&\eqd& \dlim_{t\rightarrow\infty} \left(\frac{a_{t+1}}{a_t}\left(a_t X_t-\xi_t\right) + \frac{a_{t+1}}{a_t}\xi_t-\xi_{t+1}\right)
\end{eqnarray*}
and $\left(a_t X_t-\xi_t\right)\cond Y$ as $t\to\infty$ give the result.
\end{proof}

\begin{lemma}\label{lemma: levy measure conv to zero}
Fix $c\in\{0,\infty\}$.  Let $M$ be a measure on $\mathbb R^d$ satisfying \eqref{eq: integ cond}, let $A_t$ be defined by \eqref{eq: trunc second moment func Levy}, and let $a_t$ be defined by \eqref{eq: a for Levy}. Assume that 
$A_\bullet\in MRV^c_0(B)$ for some $B\ne0$ and
\begin{eqnarray}\label{eq: Mt}
M_t(D) = t\int_{\mathbb R^d}1_A(a_t x)M(\rd x), \qquad D\in\mathfrak B(\mathbb R^d).
\end{eqnarray}
If, for $\eta\in[0,2)$,
\begin{eqnarray}\label{eq: eta moment}
\int_{|x|>1}|x|^\eta M(\rd x)<\infty
\end{eqnarray}
then $\lim_{t\rightarrow c}\int_{|x|>s}|x|^\eta M_t(\rd x) = 0$ for all $s>0$. Moreover, when $c=\infty$ \eqref{eq: eta moment} holds for every $\eta\in[0,2)$.
\end{lemma}

\begin{proof}
Let
$$
U(u) := \int_{|x|\le u}|x|^2 M(\rd x) = \tr A_u
$$
and
$$
U^t(u) := \int_{|x|\le u}|x|^2 M_t(\rd x) = ta_t^2 U(u/a_t).
$$
Note that $U\in RV^c_0$, $a_\bullet\in RV^c_{-1/2}$, and $\lim_{t\to c}a_t=1/c$. By Fubini's Theorem and the fact that $t\sim h(1/(a_t\sqrt k))= [ka_t^2 U(1/(a_t\sqrt k))]^{-1}\sim [ka_t^2 U(1/a_t)]^{-1}$ as $t\to c$ it follows that for any $s>0$
\begin{eqnarray*}
\lim_{t\to c}\int_{|x|>s}|x|^\eta M_t(\rd x) &=& \lim_{t\to c}\left[(2-\eta)\int_{s}^\infty u^{\eta-3} U^t(u) \rd u - s^{\eta-2}U^t(s)\right] \\
&=& \lim_{t\to c} t a_t^2  \left[(2-\eta)\int_{s}^\infty u^{\eta-3} U(u/a_t) \rd u \right.\\
&&\qquad \left. - s^{\eta-2}U(s/a_t)\right]\\
&=& \lim_{t\to c} k^{-1}\left[(2-\eta)\frac{\int_{s}^\infty u^{\eta-3} U(u/a_t) \rd u}{U(1/a_t)} \right.\\
&&\qquad\left.- s^{\eta-2}\frac{U(s/a_t)}{U(1/a_t)}\right]\\
&=& \lim_{t\to c} k^{-1}\left[(2-\eta)\frac{\int_{s/a_t}^\infty u^{\eta-3} U(u) \rd u}{U(s/a_t)(s/a_t)^{\eta-2}}s^{\eta-2} \right.\\
&&\qquad\left.- s^{\eta-2}\frac{U(s/a_t)}{U(1/a_t)}\right]\\
&=& k^{-1}\left(s^{\eta-2} - s^{\eta-2}\right) = 0,
\end{eqnarray*}
where the fourth equality follows by change of variables and the fifth by Karamata's Theorem (see e.g.\ Theorem 2.1 in \cite{Resnick:2007}). Note that Karamata's Theorem still holds when $U\in RV^0_0$ since, in that case, $U(1/\bullet)\in RV^\infty_0$.  We now prove the last statement. Fubini's theorem implies that for any $s>0$
$$
\int_{|x|>s}|x|^\eta M_t(\rd x) = (2-\eta)\int_{s}^\infty u^{\eta-3} U^t(u) \rd u - s^{\eta-2}U^t(s).
$$
The right side is finite by Lemma 2 on Page 277 in \cite{Feller:1971}, and hence the left side must be finite as well. 
\end{proof}

\begin{proof}[Proof of Theorem \ref{thrm: doa of normal for Levy}]
Let $M_t$ be given by \eqref{eq: Mt}, this is the L\'evy measure of $a_tX_t-\xi_t$. 

First assume that $A_\bullet\in MRV^c_0(kB)$ and that $a_t$ is given by \eqref{eq: a for Levy}. This implies that $a_\bullet\in RV^c_{-1/2}$ and
\begin{eqnarray*}
\lim_{t\rightarrow c}\int_{|x|\le1}xx^TM_t(\rd x) &=& \lim_{t\rightarrow c}ta^2_t\int_{|x|\le1/a_t}xx^TM(\rd x)\\
&=& \lim_{t\rightarrow c}k^{-1}\frac{\int_{|x|\le1/a_t}xx^TM(\rd x)}{\int_{|x|\le 1/(a_t\sqrt k)}|x|^2M(\rd x)}\\
&=& \lim_{t\rightarrow c}k^{-1}\frac{\int_{|x|\le1/a_t}xx^TM(\rd x)}{\int_{|x|\le1/a_t}|x|^2M(\rd x)}=B.
\end{eqnarray*}
From here the result follows by Lemmas \ref{lemma: levy measure conv to zero} and \ref{lemma: conv ID}.

Now assume that \eqref{eq: long time for gaus Levy} holds. Lemma \ref{lemma: conv ID} implies that for every $s>0$ 
$$
\lim_{t\to c}M_t(x:|x|>s)=0.
$$
This means that we can use Lemma \ref{lemma: eqiv between deltas}, which combined with Lemma \ref{lemma: conv ID} says that for any $s>0$
$$
\lim_{t\rightarrow c} ta_t^2\int_{|x|\le s/a_t}xx^TM(\rd x) = \lim_{t\rightarrow c} \int_{|x|\le s}xx^TM_t(\rd x)=B.
$$
Thus, for any $s>0$, $\lim_{t\rightarrow c} ta_t^2 U(s/a_t) = \tr B$, where $U(t) = \int_{|x|\le t}|x|^2M(\rd x)$. Lemma \ref{lemma: at conv to 0 or infty} implies that the sequential criterion for regular variation of monotone functions (see e.g.\ Propositions 2.3 in \cite{Resnick:2007}) holds, and hence $U \in RV^c_0$. Note that when $c=0$ we can use the sequential criterion because $U\in RV^0_0$ if and only if $U(1/\bullet)\in RV^\infty_0$. The fact that 
$$
\lim_{t\rightarrow c}\frac{\int_{|x|\le t}xx^TM(\rd x)}{\int_{|x|\le t}|x|^2M(\rd x)} = \lim_{t\rightarrow c}\frac{ta^2_t\int_{|x|\le 1/a_t}xx^TM(\rd x)}{ta_t^2\int_{|x|\le 1/a_t}|x|^2M(\rd x)} =\frac{B}{\tr B},
$$
gives $A_\bullet\in MRV^c_0(kB)$.
\end{proof}

The following technical result is easily verified.

\begin{lemma}\label{lemma: tech}
If $\int_{\mathbb R^d}|x|^2\mu(\rd x)=\infty$ and $\int_{\mathbb R^d}|x|\mu(\rd x)<\infty$ then
\begin{eqnarray*}
\lim_{t\to\infty}\frac{\int_{|x|\le t}xx^T\mu(\rd x)}{\int_{|x|\le t}|x|^2\mu(\rd x)}&=&\lim_{t\to\infty}\frac{\int_{|x|\le t}xx^T\mu(\rd x)-\int_{|x|\le t}x\mu(\rd x)\int_{|x|\le t}x^T\mu(\rd x)}{\int_{|x|\le t}|x|^2\mu(\rd x)}\\
&=&\lim_{t\to\infty}\frac{\int_{|x|\le t}xx^T\mu(\rd x)-\int_{|x|\le t}x\mu(\rd x)\int_{|x|\le t}x^T\mu(\rd x)}{\int_{|x|\le t}|x|^2\mu(\rd x)-\int_{|x|\le t}x^T\mu(\rd x)\int_{|x|\le t}x\mu(\rd x)}
\end{eqnarray*}
so long as at least one of the limits exists.
\end{lemma}

\begin{proof}[Proof of Theorem \ref{thrm: doa normal for CLT}.] 
By Corollary 3.2.15 in \cite{Meerschaert:Scheffler:2001} \eqref{eq: long time for gaus CLT} holds if and only if for any $\epsilon>0$
\begin{eqnarray}\label{eq: mu to zero}
n\mu(\{x:|x|>\epsilon/a_n\})\to 0
\end{eqnarray}
and
\begin{eqnarray}\label{eq: trunc conv matrix}
\lim_{n\to\infty}n a_n^2\left[\int_{|x|\le 1/a_t}xx^T\mu(\rd x)-\int_{|x|\le 1/a_t}x\mu(\rd x)\int_{|x|\le 1/a_t}x^T\mu(\rd x)\right]=B.
\end{eqnarray}
We will show that this is equivalent to $A_\bullet\in MRV^\infty_0(kB)$.

We begin with the case when $\int_{\mathbb R^d}|x|^2 \mu(\rd x)=\infty$. First assume that $A_\bullet\in MRV^\infty_0(kB)$ and let $a_n$ be defined by \eqref{eq: a CLT}.  In this case, Lemma \ref{lemma: levy measure conv to zero} implies that \eqref{eq: mu to zero} holds and $\int_{\mathbb R^d}|x|\mu(\rd x)<\infty$.
By Lemma \ref{lemma: tech}
\begin{eqnarray*}
&&\lim_{n\to\infty}n a_n^2\left[\int_{|x|\le 1/a_n}xx^T\mu(\rd x)-\int_{|x|\le 1/a_n}x\mu(\rd x)\int_{|x|\le 1/a_n}x^T\mu(\rd x)\right]\\
&&=\lim_{n\to\infty}k^{-1}\frac{\int_{|x|\le 1/a_n}xx^T\mu(\rd x)-\int_{|x|\le 1/a_n}x\mu(\rd x)\int_{|x|\le 1/a_n}x^T\mu(\rd x)}{\int_{|x|\le 1/ (a_n\sqrt k)}|x|^2\mu(\rd x)}\\
&&=\lim_{n\to\infty}k^{-1}\frac{\int_{|x|\le 1/a_n}xx^T\mu(\rd x)-\int_{|x|\le 1/a_n}x\mu(\rd x)\int_{|x|\le 1/a_n}x^T\mu(\rd x)}{\int_{|x|\le 1/a_n}|x|^2\mu(\rd x)}\\
&&=\lim_{n\to\infty}k^{-1}\frac{A_{1/a_n} }{\tr A_{1/a_n}}=B.
\end{eqnarray*}
Conversely, assume that \eqref{eq: long time for gaus CLT} holds for some sequence $a_n$. From univariate results (see e.g.\ \cite{Feller:1971}), it follows that $a_n\to0$, $a_n/a_{n+1}\to1$, and $\int_{\mathbb R^d}|x|\mu(\rd x)<\infty$. From \eqref{eq: trunc conv matrix} we get
\begin{eqnarray*}
\tr B &=& \lim_{n\to\infty}n a_n^2\left[\int_{|x|\le 1/a_t}|x|^2\mu(\rd x)-\int_{|x|\le 1/a_t}x^T\mu(\rd x)\int_{|x|\le 1/a_t}x\mu(\rd x)\right]\\
&=& \lim_{n\to\infty}n a_n^2\int_{|x|\le 1/a_t}|x|^2\mu(\rd x).
\end{eqnarray*}
From here, by agruments similar to those in the proof of Theorem \ref{thrm: doa of normal for Levy}, we get $\int_{|x|\le \bullet}xx^T\mu(\rd x)\in MRV^\infty_0(kB)$. By Lemma \ref{lemma: tech} 
this implies that $A_\bullet\in MRV^\infty_0(kB)$ as well.

Now consider the case when $\int_{\mathbb R^d}|x|^2 \mu(\rd x)<\infty$. Let
$$
B'= \frac{\int_{\mathbb R^d}xx^T\mu(\rd x)-\int_{\mathbb R^d}x\mu(\rd x)\int_{\mathbb R^d}x^T\mu(\rd x)}{\int_{\mathbb R^d}|x|^2\mu(\rd x)-\int_{\mathbb R^d}x^T\mu(\rd x)\int_{\mathbb R^d}x\mu(\rd x)}.
$$
By dominated convergence $A_\bullet\in MRV^\infty_0(B')$. Now let $a_n\sim c^{1/2}n^{-1/2}$ where $c^{1/2}=\eta^{-1/2}\left[\int_{\mathbb R^d}|x|^2\mu(\rd x)-\int_{\mathbb R^d}x^T\mu(\rd x)\int_{\mathbb R^d}x\mu(\rd x)\right]^{-1/2}$ for some $\eta>0$. By Markov's inequality, for any $\epsilon>0$
\begin{eqnarray*}
\lim_{n\to\infty}n\mu(x:|x|>\epsilon/a_n) &\le& \lim_{n\to\infty} n\frac{a_n^2}{\epsilon^2}\int_{|x|>\epsilon/a_n}|x|^2\mu(\rd x)\\
&=& \lim_{n\to\infty} \frac{c}{\epsilon^2}\int_{|x|>\epsilon/a_n}|x|^2\mu(\rd x)=0.
\end{eqnarray*}
By dominated convergence
\begin{eqnarray*}
&&\lim_{n\to\infty}n a_n^2\left[\int_{|x|\le 1/a_n}xx^T\mu(\rd x)-\int_{|x|\le 1/a_n}x\mu(\rd x)\int_{|x|\le 1/a_n}x^T\mu(\rd x)\right]\\
&&=\eta^{-1}\frac{\int_{\mathbb R^d}xx^T\mu(\rd x)-\int_{\mathbb R^d}x\mu(\rd x)\int_{\mathbb R^d}x^T\mu(\rd x)}{\int_{\mathbb R^d}|x|^2\mu(\rd x)-\int_{\mathbb R^d}x^T\mu(\rd x)\int_{\mathbb R^d}x\mu(\rd x)}=\eta^{-1}B'.
\end{eqnarray*} 
This implies that \eqref{eq: long time for gaus CLT} holds for $B=\eta^{-1}B'$ and some sequence $\xi_n$. Since $\tr B'=1$, $\eta^{-1}=\tr B$ and the result follows.
\end{proof}

\begin{proof}[Proof of Corollary \ref{cor: nec and suf CLT}] It suffices to show that $A_\bullet\in MRV^\infty_0(B)$ if and only if $A_\bullet'\in MRV^\infty_0(B')$. If $\int_{\mathbb R^d}|x|^2 \mu(\rd x)<\infty$ the result is immediate.  When $\int_{\mathbb R^d}|x|^2 \mu(\rd x)=\infty$ the result follows from Lemma 
\ref{lemma: tech}. The fact that, in both directions, $\int_{\mathbb R^d}|x| \mu(\rd x)<\infty$ can be shown as in the proof of Theorem \ref{thrm: doa normal for CLT}.
\end{proof}

\end{document}